\documentclass[11pt,a4paper]{article}

\title{\bf Harnack Inequalities for SDEs with Multiplicative Noise and Non-regular Drift}
\author{Huaiqian Li$^a$\footnote{Email: huaiqianlee@gmail.com.}
\quad  Dejun Luo$^b$\footnote{Email: luodj@amss.ac.cn. Supported in
part by the Key Laboratory of RCSDS, CAS (2008DP173182), NSFC
(11101407) and AMSS (Y129161ZZ1)} \quad Jian
Wang$^c$\footnote{Email: jianwang@fjnu.edu.cn. Supported in part by
NSFC (11201073) and the Program for New Century Excellent Talents in
Universities of Fujian (JA12053)}
\vspace{3mm}\\
{\footnotesize $^a$School of Mathematics, Sichuan University, Chengdu 610064, China}\\
{\footnotesize $^b$Institute of Applied Mathematics, Academy of Mathematics and Systems Science,}\\
{\footnotesize Chinese Academy of Sciences, Beijing 100190, China}\\
{\footnotesize $^c$School of Mathematics and Computer Science, Fujian Normal University,
Fuzhou 350007, China}
}
\date{}

\usepackage{amssymb,amsmath,amsfonts,amsthm,array,bm,color}

\def\R{\mathbb{R}}
\def\E{\mathbb{E}}
\def\P{\mathbb{P}}
\def\Id{\hbox{\rm Id}}
\def\B{\mathcal{B}}

\def\F{\mathcal{F}}

\newcommand{\ra}{\rightarrow}

\def\Tr{\textup{Tr}}
\def\d{\textup{d}}

\def\<{\langle}
\def\>{\rangle}

\newtheorem{theorem}{Theorem}[section]
\newtheorem{lemma}[theorem]{Lemma}       

\newtheorem{proposition}[theorem]{Proposition}

\begin{document}

\maketitle

\makeatletter 
\renewcommand\theequation{\thesection.\arabic{equation}}
\@addtoreset{equation}{section}
\makeatother 

\vspace{-6mm}

\begin{abstract}
The log-Harnack inequality and Harnack inequality with powers for semigroups
associated to SDEs with non-degenerate diffusion coefficient and non-regular
time-dependent drift coefficient are established, based on the recent papers
\cite{Flandoli, Zhang11}. We consider two cases in this work: (1) the drift fulfills
the LPS-type integrability, and (2) the drift is uniformly H\"older continuous
with respect to the spatial variable. Finally, by using explicit heat kernel
estimates for the stable process with drift, the Harnack inequality for the stochastic differential equation driven by symmetric stable
process is also proved.
\end{abstract}



\section{Introduction and Main Results}
The dimension-free Harnack inequality with powers introduced in
\cite{Wang97} and the  log-Harnack inequality introduced in
\cite{Rockner} have been intensively investigated for various
stochastic (partial) differential equations. They are efficiently
applied to study heat kernel estimates, functional inequalities,
transportation-cost inequalities and properties of invariant
measures, see e.g.\ \cite{Wang-book} and references therein.
Consider the following stochastic differential equation (SDE) on
$\R^d$:
  \begin{equation}\label{SDE}
  \d X_t=\sigma(t,X_t)\,\d W_t+b(t,X_t)\,\d t,\quad X_0=x,
  \end{equation}
where $\sigma:[0,\infty)\times\R^d\ra\R^d\otimes\R^d$ and $b:[0,\infty)\times\R^d\to \R^d$ are two
Borel measurable functions, and $(W_t)_{t\ge0}$ is a standard $d$-dimensional
Brownian motion defined on a complete filtered probability space $(\Omega, \F,\P, (\F_t)_{t\geq 0})$.
When the equation \eqref{SDE} has a unique solution
for any starting point $x$, we denote it by $X_t(x)$ and define the
associated Markov semigroup $(P_t)_{t\ge0}$ as follows:
  $$P_tf(x)=\E f(X_t(x)), \quad t\ge0, x\in\R^d, f\in \B_b(\R^d).$$
If the coefficients $\sigma$ and $b$ are semi-Lipschitz continuous with
respect to the spatial variable locally uniformly in the time variable,
Harnack inequalities for $P_t$ have been established in
\cite{Wang11}; see Theorem \ref{sect-2-thm} below for the explicit
statement. Recently, Harnack inequalities for \eqref{SDE} with
log-Lipschitz continuous coefficients have been studied in
\cite{SWY}. The aim of this paper is to consider Harnack
inequalities for SDE \eqref{SDE} with non-regular time-dependent drift
coefficient.

\subsection{Drift satisfying the LPS-type condition}
We first consider the case where the drift $b$ satisfies an integrability condition,
which is known in fluid dynamics as the Ladyzhenskaya--Prodi--Serrin condition
(LPS condition for short). More precisely, assume that the drift coefficient
$b\in L_{loc}^q\big([0,\infty),L^p(\R^d)\big) $ for some $p>d$ and $q>2$ such that
  \begin{equation}\label{condition}
  \frac dp +\frac 2q<1,
  \end{equation}
that is, for all $T>0$, it holds
  \begin{equation}\label{LPS}
  \int_0^T\bigg(\int_{\R^d}|b(t,x)|^p\,\d x\bigg)^{q/p}\,\d t<+\infty.
  \end{equation}
The SDE \eqref{SDE} with diffusion coefficient $\sigma=\Id$
and drift coefficient $b$ satisfying \eqref{LPS} was first
studied by Krylov and R\"ockner \cite{KR05}, where the existence of the unique strong
solution was proved. X. Zhang \cite{Zhang05b} extended their result to the more general
SDE \eqref{SDE} with variable diffusion coefficient $\sigma$. Furthermore,
it was shown in \cite{Zhang11} (see also \cite{FF13}) that SDE \eqref{SDE} generates
a unique stochastic flow $X_t$ of homeomorphisms on $\R^d$, provided that
the diffusion coefficient $\sigma$ satisfies the following conditions:
\begin{itemize}
\item[(H$^\sigma_1$)] $\sigma(t,x)$ is uniformly continuous in $x\in\R^d$ locally uniformly
with respect to $t$, and there exist positive constants $K$ and $\delta$ such that for all
$(t,x)\in [0,\infty)\times\R^d$,
  \begin{equation}\label{non-degeneracy}
  \delta|y|^2\leq \big|\sigma(t,x)^\ast\, y\big|^2\leq K|y|^2,
  \end{equation}
where $\sigma(t,x)^\ast$ is the transposition;\footnote{
We point out that the non-degeneracy condition {\rm (H$^\sigma_1$)} in \cite{Zhang11} is
incorrect, and it reads as
  \begin{equation}\label{non-degeneracy-1}
  \delta|y|^2\leq \sum_{ik}|\sigma^{ik}(t,x) y_i|^2\leq K|y|^2.
  \end{equation}
It is clear that \eqref{non-degeneracy} implies \eqref{non-degeneracy-1}, but the
inverse implication does not hold. To see the latter, simply look at the $(2\times 2)$
constant matrix
  $$\sigma=\left(\begin{array}{cc}
  1 & -1\\
  -1 & 1
  \end{array}\right).$$
Nevertheless, the arguments of \cite{Zhang11} still work and the results all hold true. }
\item[(H$^\sigma_2$)] $|\nabla\sigma(t,\cdot)|\in L_{loc}^q\big([0,\infty),L^p(\R^d)\big)$ with
the same $p,q$ as in \eqref{condition}, where $\nabla$ denotes the generalized gradient with respect to $x$.
\end{itemize}

In the recent paper \cite{BFGM}, Beck et al.\ considered SDE \eqref{SDE} with
$\sigma=\Id$ and a drift $b$ satisfying the generalized LPS condition (i.e., `$<$' in
\eqref{condition} is replaced by `$\leq$'). They first established the well-posedness of the
corresponding stochastic continuity (and also transport) equation, from which they deduced
the existence of a unique Lagrangian flow associated to \eqref{SDE}. Notice that,
however, in the limit case $(p,q)=(d,\infty)$, they assumed in addition that
the $L_{loc}^\infty\big([0,\infty),L^d(\R^d)\big)$-norm of the drift $b$ is small enough (see
\cite[Condition 8]{BFGM}), thus leaving the general case as an open question.

Inspired by X. Zhang's work \cite{Zhang11}, J. Shao established in \cite[Theorem 2.1]{Shao} the
Harnack inequalities for SDE \eqref{SDE} by using the coupling method, under some
additional assumptions (see (H$^\sigma_3$), (H$^\sigma_4$) and (H$^b$) in \cite{Shao}).
However, there are some extra constants on the
right hand sides of the inequalities \cite[(2.3) and (2.4)]{Shao}. In the next theorem,
we remove the additional constant in the log-Harnack inequality \cite[(2.3)]{Shao};
moreover, we do not need the extra conditions (H$^\sigma_3$) and (H$^b$)
in \cite{Shao}.

\begin{theorem}\label{thm-log-Harnack}
Assume that $\sigma$ fulfills {\rm (H$^\sigma_1$)} and {\rm (H$^\sigma_2$)},
and $b\in L^q_{loc}\big([0,\infty),L^p(\R^d)\big)$ with $p,q$ verifying \eqref{condition}.
Let $P_{t}$ be the semigroup associated to \eqref{SDE}. Then, for any $T>0$, there is a positive constant $C>0$ such that the following log-Harnack inequality
  \begin{equation}\label{Log-Harnack-1}
  P_{t}\log f(y)\leq \log P_{t}f(x)+\frac{C|y-x|^2}{\delta(t-s)},\quad x,y\in\R^d
  \end{equation}
holds for $0<t\le T$ and $f\in \B_b(\R^d)$ with $f\geq 1$, where $\delta$ is the constant
in {\rm (H$^\sigma_1$)}.
\end{theorem}

Our method is based on the $L^2$-gradient estimate of the semigroup given in the proof of
\cite[Theorem 3.5]{Zhang11} (see also \eqref{gradient-estimate-1} of the current paper).
For the moment, we are unable to remove the extra constant in the Harnack inequality
\cite[(2.4)]{Shao}, since we do not have the $L^1$-gradient estimate.

\subsection{H\"older continuous drift}

Next we consider the case where the drift coefficient $b$ is H\"older continuous
with respect to the spatial variable. The motivation for considering this type of
drift comes from the papers \cite{FGP, Flandoli} of Flandoli, Gubinelli and Priola.
In the influential work \cite{FGP}, the authors considered the SDE \eqref{SDE}
with diffusion coefficient $\sigma= \Id$ and drift $b\in L_{loc}^\infty\big([0,\infty),
C_b^\theta(\R^d,\R^d)\big)$ for some $\theta\in(0,1)$.
They proved that, in this case, equation \eqref{SDE}
generates a stochastic flow of diffeomorphisms, from which they constructed an
explicit solution to the corresponding stochastic transport equation. Flandoli
et al.\ extended in \cite{Flandoli} the property of flow of diffeomorphisms to
more general SDEs with non-constant diffusion coefficient $\sigma$, satisfying
a uniform non-degeneracy. Their proof is based on a modified Zvonkin transformation
(called the It\^o--Tanaka trick in \cite{Flandoli}).
Notice that the main part of \cite{Flandoli} is focused
on the SDE \eqref{SDE} with coefficients independent on time, but the authors
mentioned in \cite[Remark 9]{Flandoli} that their method works as well in the
time-dependent case; moreover, they outlined the essential steps needed for transforming
the proofs to the time-dependent case.

Here are our assumptions in this case (see Subsection 3.1 for the definition of the
functional spaces):
\begin{itemize}
\item[(H3)] $b\in L_{loc}^\infty\big([0,\infty), C_b^\theta(\R^d,\R^d)\big)$ and $\sigma\in
L_{loc}^\infty\big([0,\infty), C_b^{1+\theta}(\R^d,\R^d\otimes\R^d)\big)$ for some constant
$\theta\in(0,1)$;
\item[(H4)] for any $(t,x)\in[0,\infty)\times\R^d$, the inverse of $a(t,x):=\sigma(t,x)
\sigma(t,x)^\ast$ exists, and
  $$\sup_{t\in [0,T],\, x\in\R^d}\|a^{-1}(t,x)\|_{HS}<\infty \quad\mbox{for all } T>0,$$
where $\|\cdot\|_{HS}$ is the Hilbert--Schmidt norm of matrices.
\end{itemize}

Under these conditions we shall prove

\begin{theorem}\label{thm-Harnack}
Assume the hypotheses {\rm (H3)} and {\rm (H4)}. Let $P_t$ be the semigroup associated
to the It\^o SDE \eqref{SDE}. Then, for any $T>0$, there are three positive constants $K$, $\kappa$ and
$\delta$ such that
\begin{itemize}
\item[\rm(1)] for any $0<t\le T$ and $f\in \B_b(\R^d)$ with $f\geq 1$, it holds
  \begin{equation}\label{log-Harnack-time}
  P_t\log f(y)\leq \log P_tf(x)+\frac{2K|x-y|^2}{\kappa^2(1-e^{-Kt})}
  \quad\mbox{for all } x,y\in\R^d;
  \end{equation}
\item[\rm(2)] for $p>(1+\delta/\kappa)^2$ and $\delta_p:=\max\{\delta,\kappa(\sqrt{p}-1)/2\}$, it holds
  \begin{equation}\label{Harnack-ineq-time}
  (P_t f(y))^p\leq (P_t f^p(x))\exp\bigg[\frac{K\sqrt{p}\,(\sqrt{p}-1)|x-y|^2}
  {\delta_p[(\sqrt{p}-1)\kappa-\delta_p](1-e^{-Kt})}\bigg]
  \end{equation}
for all $f\in\mathcal B^+_b(\R^d)$, $x,y\in\R^d$ and $0<t\le T$.
\end{itemize}
\end{theorem}

Unfortunately, the explicit expressions of the constants $K,\kappa$ and $\delta$ are a
little complicated, as can be seen from the proof in Section 3.
To point out the main difference between the H\"{o}lder continuous
situation and the Lipschitz continuous setting, we first explain the
idea of coupling for \eqref{SDE} with semi-Lipschitz continuous
drift. For simplicity, we consider the time-independent case where $\sigma=\Id$
and for some $K\in\R$,
  $$\langle b(x)-b(y),x-y\rangle\le K|x-y|^2,\quad x,y\in\R^d.$$
For $x\neq y\in\R^d$ and $T>0$, let $X_t$ solve \eqref{SDE} with $X_0=x$, and $Y_t$ solve
  $$\d Y_t=\d B_t+b(Y_t)\,\d t+\frac{(X_t-Y_t)|x-y|e^{-Kt}}{|X_t-Y_t|\int_0^Te^{-2Ks}\,\d s}\,\d t,\quad Y_0=y.$$
Then, $Y_t$ is well defined up to the coupling time
  $$\tau=\inf\{t\ge0:X_t=Y_t\}.$$ Let $Y_t=X_t$ for $t\ge \tau.$
We have
  $$\d|X_t-Y_t|\le K|X_t-Y_t|\d t-\frac{|x-y|e^{-Kt}}{\int_0^Te^{-2Ks}\,\d s}\,\d t, \quad
  t\le \tau.$$
That is,
  $$\d(|X_t-Y_t|e^{-Kt})\le-\frac{|x-y|e^{-2Kt}}{\int_0^Te^{-2Ks}\,\d s}\,\d t
  , \quad t\le \tau.$$
This implies $\tau\le T$ and hence, $X_T=Y_T$. Combining it with the
Girsanov theorem yields the desired Harnack inequalities; see for instance
the proof of \cite[Theorem 2]{ArnaudonThalmaierWang2006} or that of
\cite[Theorem 1.1]{Wang11}. However, due to the poor H\"{o}lder
regularity of the drift vector field $b$, it seems that in the
present setting one cannot directly use the coupling method above to
establish the Harnack inequalities.

Now we briefly describe our strategy to help the readers understand better
the proof of Harnack inequalities with H\"{o}lder continuous drift.
Following the ideas in the proof of \cite[Theorem 7]{Flandoli}, we
can transform the equation \eqref{SDE} into a new SDE \eqref{conjugate-SDE} which has
smooth coefficients with bounded derivatives; moreover, there is a simple
relationship between their corresponding semigroups (see \eqref{relationship} below).
For this new equation \eqref{conjugate-SDE}, we can check that the assumptions (A1)--(A3)
in \cite{Wang11} are satisfied under our hypotheses (H3)--(H4). In
this way we first get Harnack inequalities for the semigroup
associated with the new equation \eqref{conjugate-SDE}, then the
relationship \eqref{relationship} between the semigroups allows us to
prove Theorem \ref{thm-Harnack}.

This paper is organized as follows. By making use of the $L^2$-gradient estimate in
\cite[p.1109]{Zhang11}, we establish in Section 2 the log-Harnack inequality
\eqref{Log-Harnack-1} by applying the semigroup interpolation scheme (see e.g. \cite{BGL}
for intensive studies on Markov Triples) and the
Zvonkin transformation. In Section 3, we first recall some necessary
results from the references \cite{Flandoli, Wang11}, then the main part is devoted to
check that the coefficients of the transformed SDE \eqref{conjugate-SDE} verify the hypotheses
(A1)--(A3) in \cite{Wang11}. With the key relation \eqref{relationship} in hand, it is easy to give
the proof of Theorem \ref{thm-Harnack}. Finally, by using explicit heat kernel estimates,
we establish in Section 4 the Harnack inequality for the SDE driven by $\alpha$-stable process.

\section{Log-Harnack inequality for SDE with LPS-type drift}

This section is devoted to the proof of Theorem \ref{thm-log-Harnack}. We shall first
prove the log-Harnack inequality \eqref{Log-Harnack-1} for the semigroup associated to the following
It\^o SDE without drift:
  \begin{equation}\label{SDE.1}
  \d Y_t=\sigma(t,Y_t)\,\d W_t,\quad Y_0=x,
  \end{equation}
where $\sigma$ verifies {\rm (H$^\sigma_1$)} and {\rm (H$^\sigma_2$)}. In the sequel, we denote
by $T_{s,t}$ the two-parameter semigroup associated to \eqref{SDE.1} defined by
  $$T_{s,t}f(x)=\E\big(f(Y_t)|Y_s=x\big),\quad 0\le s\le t.$$
For simplicity, we set $T_tf(x)=T_{0,t}f(x)$. For $t\geq 0$, define the time-dependent
second order differential operator associated with $Y_t$ as follows
  $$L_t f(x)=\frac12\Tr\big[a(t,x)\nabla ^2 f(x)\big]
  ,\quad f\in C_b^2(\R^d),$$
where $a(t,x)=\sigma(t,x)\sigma(t,x)^\ast$ and $\nabla ^2 f$ is the Hessian matrix of $f$.
Then we have the well-known Kolmogorov equations:
  $$\partial_s T_{s,t}f=-L_sT_{s,t}f,\quad \partial_tT_{s,t}f=T_{s,t}L_tf,$$
where $\partial_s=\frac{\partial}{\partial s}$. For $f,g\in C^2(\R^d)$, define
  \begin{align*}
  \Gamma(t)(f,g)&=\frac{1}{2}\{L_t(fg)-gL_tf-fL_tg\},
  \end{align*}
and set $\Gamma(t)(f)=\Gamma(t)(f,f)$ for short.
Then
  $$\Gamma(t)(f,g)=\frac12\big\<\sigma(t,x)^\ast\,\nabla f, \sigma(t,x)^\ast\,\nabla g\big\>.$$

Let $\{\rho_n\}_{n\geq1}$ be a family of mollifiers on $\R^d$ and set $\sigma^n(t,x)=
(\sigma(t,\cdot)\ast\rho_n)(x),\,n\geq 1$. We consider the following It\^o SDE with smooth coefficient:
  \begin{equation*}\label{SDE-smooth}
  \d Y^n_t=\sigma^n(t,Y^n_t)\,\d W_t,\quad Y^n_0=x.
  \end{equation*}
Following the arguments in the proof of \cite[Theorem 3.5]{Zhang11}, we can show that
  \begin{equation}\label{gradient-estimate}
  C_1:=\sup_{n\geq 1}\sup_{t\leq T}\sup_{x\in\R^d}\E\big(|\nabla Y^n_t(x)|^2\big)<+\infty,\quad T>0.
  \end{equation}
Then for any $f\in C_b^1(\R^d)$, $x,y\in\R^d$ and $t>0$, by the mean value formula,
  $$f(Y^n_t(x))-f(Y^n_t(y))=\int_0^1\big\<\nabla f\big[Y^n_t(y+r(x-y))\big],
  \big[\nabla Y^n_t(y+r(x-y))\big](x-y)\big\>\,\d r.$$
Therefore,
  \begin{align*}
  |f(Y^n_t(x))-f(Y^n_t(y))|&\leq |x-y|\int_0^1\big|\nabla f\big[Y^n_t(y+r(x-y))\big]\big|
  \cdot\big|\nabla Y^n_t(y+r(x-y))\big|\,\d r.
  \end{align*}
Cauchy's inequality and \eqref{gradient-estimate} imply that for any $f\in C_b^1(\R^d)$, $x,y\in\R^d$ and $0<t\le T,$
  \begin{equation}\label{rrr}
  \big|\E f(Y^n_t(x))-\E f(Y^n_t(y))\big|\leq \sqrt{C_1}\,|x-y|\int_0^1\Big(\E\big|\nabla f\big[Y^n_t(y+r(x-y))\big]\big|^2\Big)^{1/2}\,\d r.
  \end{equation}
Moreover, by \cite[(3.7)]{Zhang11}, we have $$\lim_{n\to\infty}\E|Y^n_t(x)-Y_t(x)|=0.$$ Thus,
by the dominated convergence theorem, letting $n$ tend to $\infty$ in \eqref{rrr}
yields that for any $f\in C_b^1(\R^d)$, $x,y\in\R^d$ and $0<t\le T,$
  $$\big|\E f(Y_t(x))-\E f(Y_t(y))\big|\leq \sqrt{C_1}\,|x-y|\int_0^1\Big(\E\big|\nabla f\big[Y_t(y+r(x-y))\big]\big|^2\Big)^{1/2}\,\d r.$$
Now we let $y\to x$ and obtain
  \begin{equation*}
  |\nabla T_tf(x)|^2\leq C_1T_t|\nabla f|^2(x),\quad x\in\R^d, 0<t\le T.
  \end{equation*}
Similarly, we have for all $0\le s\leq t\le T$,
  \begin{equation}\label{gradient-estimate-1}
  |\nabla T_{s,t}f(x)|^2\leq C_1T_{s,t}|\nabla f|^2(x),\quad x\in\R^d.
  \end{equation}

Now standard arguments lead to the log-Harnack inequality for the semigroup $T_{s,t}$.

\begin{proposition}\label{3-prop-1}
Assume that $\sigma$ verifies {\rm (H$^\sigma_1$)} and {\rm (H$^\sigma_2$)}. Then for any $T>0$,
there is a constant $C_1>0$ such that for all $f\in \B_b(\R^d)$ with $f\geq 1$,
  \begin{equation}\label{Log-Harnack}
  T_{s,t}\log f(y)\leq \log T_{s,t}f(x)+\frac{C_1|y-x|^2}{2\delta(t-s)},
  \quad x,y\in \R^d,\,s\leq t,
  \end{equation}
where $\delta$ is the constant in {\rm (H$^\sigma_1$)}.
\end{proposition}

\begin{proof}
Take $f\geq 1$. Applying It\^{o}'s formula, we have
  \begin{eqnarray*}
  \d \log T_{u,t}f(Y_{u})&=&\< \nabla\log T_{u,t}f(Y_{u}), \sigma(u,Y_{u})\,\d W_u \>
  +L_u \log T_{u,t}f(Y_{u})\,\d u -\frac{L_u T_{u,t}f(Y_{u})}{T_{u,t}f(Y_{u})}\,\d u\\
  &=&\< \nabla\log T_{u,t}f(Y_{u}), \sigma(u,Y_{u})\,\d W_u \>
  - \frac{\Gamma(u)(T_{u,t}f)(Y_{u})}{(T_{u,t}f)^2(Y_u)}\,\d u,
  \end{eqnarray*}
where the last equality follows by
  $$L_u\log T_{u,t}f=\frac{L_u T_{u,t}f}{T_{u,t}f}-\frac{\Gamma(u)(T_{u,t}f)}{(T_{u,t}f)^2}.$$
Then, by integrating from $s$ to $u$, we get
  \begin{equation*}
  \begin{split}
  \log T_{u, t}f(Y_{u}) - \log T_{s,t}f(Y_s)
  &=\int_s^u\< \nabla\log T_{r,t}f(Y_{r}), \sigma(r,Y_{r})\,\d W_r \>  -\int^{u}_s\frac{\Gamma(r)(T_{r,t}f)(Y_{r})}{(T_{r,t}f)^2(Y_r)}\,\d r.
  \end{split}
  \end{equation*}
Taking expectation with respect to $\{Y_s=x\}$, we have
  \begin{eqnarray}\label{ac}
  T_{s,u}\log T_{u, t}f(x) - \log T_{s,t}f(x)
  =-\int^{u}_s T_{s,r}\bigg(\frac{\Gamma(r)(T_{r,t}f)}{(T_{r,t}f)^2}\bigg)(x)\,\d r,\quad u\in [s,t].
  \end{eqnarray}

Now for $x,y\in \R^d$, let $\gamma_u=(y-x)\frac{u-s}{t-s}+x$. The identity \eqref{ac} implies that
$[s,t]\ni u\mapsto T_{s,u}\log T_{u,t}f(\gamma_u)$ is absolutely continuous; thus by (H$^\sigma_1$)
and the definition of $\Gamma(u)$, for any $0\le s\le t\le T$, we have
  \begin{eqnarray*}
  \frac{\d}{\d u}T_{s,u}\log T_{u,t}f(\gamma_u)
  &=& -T_{s,u}\bigg(\frac{\Gamma(u)(T_{u,t}f)}{(T_{u,t}f)^2}\bigg)(\gamma_u)
  +\big\<\nabla \big(T_{s,u}\log T_{u,t}f\big)(\gamma_u),\dot{\gamma}_u\big\>\\
  &\leq& -\frac{\delta}{2} T_{s,u}|\nabla \log T_{u,t}f|^2(\gamma_u)
  +|\dot{\gamma}_u|\big(C_1T_{s,u}|\nabla \log T_{u,t}f|^2(\gamma_u)\big)^{\frac12}\\
  &\leq& \frac{C_1}{2\delta}|\dot{\gamma}_u|^2,\quad \mathcal{L}^1\mbox{-a.e. } u\in [s,t] ,
  \end{eqnarray*}
where in the first inequality we have used \eqref{gradient-estimate-1}.
Integrating from $s$ to $t$ gives us the log-Harnack inequality \eqref{Log-Harnack}.
\end{proof}

It remains to transfer the above result to the general It\^o SDE \eqref{SDE} with drift.
Before moving on, we introduce two function spaces: for $p,q\geq 1$ and $s<t$, let
  $$L_p^q(s,t)=L^q\big([s,t],L^p(\R^d)\big)\quad \mbox{and}\quad
  \mathbb H_{2,p}^q(s,t)=L^q\big([s,t],W^{2,p}(\R^d)\big), $$
where $W^{2,p}(\R^d)$ is the  standard Sobolev space.

We shall need the following preparations which are taken from \cite[pp.1110--1111]{Zhang11}.
Assume that $\sigma$ satisfies {\rm (H$^\sigma_1$)} and $b\in L_p^q(0,T)$
with $p,q$ verifying \eqref{condition} for any $T>0$. Fix $0<T_0\le T$. For any $0\leq s< t\leq T$
with $t-s\leq T_0$, let $(u(r,x))_{s\leq r\leq t}$ with $u(r,x):=\big(u^1(r,x),\ldots, u^d(r,x)\big)$
be the solution to the backward parabolic equation
  \begin{equation*}
  \partial_r u^i(r,x)+L_r u^i(r,x)+b^i(r,x)=0,\quad u^i(t,x)=0,\quad 1\leq i\leq d,
  \end{equation*} where $$L_r u^i(r,x)=\frac12\Tr\big[a(r,x)\nabla ^2 u^i(r,x)\big]+ \langle b(r,x), \nabla u^i(r,x)\rangle
  ,$$ $a(r,x)=\sigma(r,x)\sigma(r,x)^\ast$ and $b(r,x):=\big(b^1(r,x),\ldots, b^d(r,x)\big)$.
Then by \cite[Theorem 5.1]{Zhang11} (see also \cite[Theorem 10.3 and Remark 10.4]{KR05}), one has
  \begin{equation}\label{PDE.0}
  C_2:=\sup_{s\in[0\vee(t-T_0),t]}\Big(\|\partial_r u\|_{L_p^q(s,t)}+\|u\|_{\mathbb H_{2,p}^q(s,t)}\Big)<+\infty.
  \end{equation}
It follows from \cite[Lemma 10.2]{KR05} that the function $(r,x)\mapsto \nabla u(r,x)$ is
H\"older continuous and for fixed $\delta\in \big(0,\frac12-\frac d{2p}-\frac1q\big)$,
there exists a constant $C_3>0$ depending on $p,q,\delta$ and $T$ such that
  \begin{equation}\label{PDE.1}
  \sup_{(r,x)\in[s,t]\times\R^d}|\nabla u(r,x)|\leq C_3T_0^\delta.
  \end{equation}
Define $\Phi_r(x)=x+u(r,x),\,(r,x)\in [s,t]\times\R^d$. It is easy to see that
  \begin{equation}\label{PDE}
  \partial_r\Phi_r(x)+L_r\Phi_r(x)=0,\quad \Phi_t(x)=x.
  \end{equation}
Moreover, if $T_0$ is small enough, we deduce from \eqref{PDE.1} that for all $r\in[s,t]$,
  \begin{equation}\label{diffeomorphism}
  \frac12 |x-y|\leq |\Phi_r(x)-\Phi_r(y)|\leq \frac32|x-y|\quad \mbox{for all } x,y\in\R^d.
  \end{equation}
Therefore $\Phi_r$ is a diffeomorphism on $\R^d$. The following result is
proved in \cite[Lemma 4.3]{Zhang11}.

\begin{lemma}[Zvonkin transformation]\label{Zvonkin}
Let $X_r$ be an $\R^d$-valued $(\F_r)_{r\geq0}$-adapted continuous process satisfying
  $$\P\bigg\{\omega\in \Omega:\int_s^t\big(|b(r,X_r(\omega))|+|\sigma(r,X_r(\omega))|^2\big)
  \,\d r<+\infty\bigg\}=1.$$
Then $X_r$ solves the equation \eqref{SDE} on the time interval $[s,t]$ if and only if
$Y_r=\Phi_r(X_r)$ solves the following SDE on $[s,t]$:
  \begin{equation}\label{Zvonkin.1}
  \d Y_r=\Sigma(r,Y_r)\,\d W_r,
  \end{equation}
where $\Sigma(r,y)=(\nabla\Phi_r\cdot\sigma(r,\cdot))\circ \Phi^{-1}_r(y)$.
\end{lemma}

With Proposition \ref{3-prop-1} and Lemma \ref{Zvonkin} in mind,
we can now present

\begin{proof}[Proof of Theorem \ref{thm-log-Harnack}]
We first check that the matrix valued function $\Sigma$ given in Lemma \ref{Zvonkin} satisfies
{\rm (H$^\sigma_1$)} and {\rm (H$^\sigma_2$)} with $\sigma$ replaced by $\Sigma$.
To this end, we fix some $0\leq s< t\leq T$ with $t-s\leq T_0$. By \eqref{diffeomorphism},
  $$\frac12\leq |\nabla\Phi_r(x)| \leq \frac32\quad \mbox{for all }(r,x)\in[s,t]\times\R^d.$$
From the definition of $\Sigma$ and {\rm (H$^\sigma_1$)}, we deduce that
  $$\frac14 \delta|y|^2\leq |\Sigma(r,x)^\ast y|^2\leq \frac94 K|y|^2\quad \mbox{for all }
  (r,x)\in[s,t]\times\R^d \mbox{ and } y\in\R^d.$$
Thus {\rm (H$^\Sigma_1$)} holds with new constants $\frac14\delta$ and $\frac94 K$. For the second
condition {\rm (H$^\Sigma_2$)}, we note that
  $$\partial_l\Sigma^{ik}(r,y)=\big[\big(\partial_{l'}\partial_j\Phi_r^i\cdot\sigma^{jk}(r,\cdot)
  +\partial_j\Phi_r^i\cdot\partial_{l'}\sigma^{jk}(r,\cdot)\big)\circ\Phi_r^{-1}(y)\big]\cdot \partial_l\Phi_r^{-1,l'}(y).$$
By \eqref{PDE.0}, \eqref{diffeomorphism} and {\rm (H$^\sigma_2$)}, we conclude that
$\|\partial_l\Sigma^{ik}\|_{L_p^q(s,t)}<+\infty$. That is, {\rm (H$^\Sigma_2$)} also holds
on the small interval $[s,t]$.

Denote by $\tilde T_{s,t}$ the semigroup associated to the new SDE \eqref{Zvonkin.1} without drift.
We can apply Proposition \ref{3-prop-1} to obtain that, for any $f\in \B_b(\R^d)$ with $f\geq 1$
and any $0\leq s<t\leq T$ with $t-s\leq T_0$, it holds
  \begin{equation}\label{3-thm-1.1}
  \tilde T_{s,t}\log f(y)\leq \log \tilde T_{s,t}f(x)+\frac{\tilde C_1|y-x|^2}{\delta(t-s)},
  \quad x,y \in\R^d,
  \end{equation}
where $\tilde C_1>0$ is some constant. We have to transfer the above
log-Harnack inequality \eqref{3-thm-1.1} to the semigroup $P_{s,t}$ associated to
\eqref{SDE}. This process is summarized in the next result.

\begin{lemma}
For any $f\in \B_b(\R^d)$ with $f\geq 1$ and any $0\leq s<t\leq T$ with $t-s\leq T_0$,
it holds
  \begin{equation}\label{3-thm-1.2}
  P_{s,t}\log f(y)\leq \log P_{s,t}f(x)+\frac{\tilde C_1|y-x|^2}{\delta(t-s)}
  \quad \mbox{for all }x,y\in\R^d.
  \end{equation}
\end{lemma}

\begin{proof}
Fix $0\leq s<t\leq T$ with $t-s\leq T_0$. For $g\in \B_b(\R^d)$,
by the definition of the semigroup $P_{s,t}$ and Lemma \ref{Zvonkin}, we have
  \begin{align*}
  P_{s,t}\, g(x)&=\E\big(g(X_t)|X_s=x\big)=\E\big(g\big(\Phi_t^{-1}(Y_t)\big)|\Phi_s^{-1}(Y_s)=x\big)\\
  &=\E\big(g(Y_t)|Y_s=\Phi_s(x)\big)=\tilde T_{s,t}\,g\big(\Phi_s(x)\big),
  \end{align*}
where the third equality follows from \eqref{PDE}. Therefore, for $f\in \B_b(\R^d)$ with
$f\geq 1$, by \eqref{3-thm-1.1},
  \begin{align*}
  P_{s,t}\log f(y)=\tilde T_{s,t}\log f\big(\Phi_s(y)\big)
  &\leq \log \tilde T_{s,t} f\big(\Phi_s(x)\big)+\frac{\tilde C_1|y-x|^2}{\delta(t-s)}\\
  &=\log P_{s,t}f(x)+\frac{\tilde C_1|y-x|^2}{\delta(t-s)},
  \end{align*}
which is the desired inequality.
\end{proof}

We continue the proof of Theorem \ref{thm-log-Harnack}. It remains to extend the
above result to the case where $t-s>T_0$, which follows from the semigroup property.
If $t-s\in(T_0,2T_0]$, then by the semigroup property and Jensen's inequality, we have
  \begin{align*}
  P_{s,t}\log f(y)&=P_{s+T_0,t}\big(P_{s,s+T_0}\log f\big)(y)
  \leq P_{s+T_0,t}\big[\log\big(P_{s,s+T_0} f\big)\big](y)\\
  &\leq \log\big[P_{s+T_0,t}\big(P_{s,s+T_0} f\big)\big](x)+\frac{\tilde C_1|y-x|^2}{\delta(t-s)}\\
  &=\log\big(P_{s,t} f\big)(x)+\frac{\tilde C_1|y-x|^2}{\delta(t-s)},
  \end{align*}
where in the second inequality we have used \eqref{3-thm-1.2}.
We complete the proof by repeating this procedure.
\end{proof}

\section{Harnack inequalities for SDEs with H\"older continuous drift}

In this section, we present the proof of Theorem \ref{thm-Harnack}. In the first subsection, we
give some notations of spaces of spatially H\"older continuous functions and preliminary results, then
we shall prove Theorem \ref{thm-Harnack} in the second subsection.

\subsection{Notations and preliminary results}

We adopt the notations in \cite[p.7]{FGP}. Let $T>0$ and $\theta\in(0,1)$ be fixed.
Define the space $L^\infty\big([0,T],C_b^\theta(\R^d)\big)$ as the set of all
bounded Borel functions $f:[0,T]\times\R^d\to \R$ for which
  \begin{equation}\label{norm}
  [f]_{\theta,T}=\sup_{t\in[0,T]}\sup_{x\neq y\in\R^d}\frac{|f(t,x)-f(t,y)|}{|x-y|^\theta}<+\infty.
  \end{equation}
This is a Banach space with respect to the usual norm $\|f\|_{\theta,T}=\|f\|_{0,T}+[f]_{\theta,T}$
where $\|f\|_{0,T}=\sup_{(t,x)\in[0,T]\times\R^d}|f(t,x)|$ is the supremum norm.
If $f$ is vector-valued or matrix-valued, then we simply replace the absolute
value in the numerator of \eqref{norm} and in the definition of $\|f\|_{0,T}$
by the Euclidean norm or Hilbert--Schmidt norm, which gives us the
spaces $L^\infty\big([0,T],C_b^\theta(\R^d, \R^d)\big)$ and
$L^\infty\big([0,T],C_b^\theta(\R^d,\R^d\otimes\R^d)\big)$ respectively.

Moreover, for $n\geq1$, $f\in L^\infty\big([0,T],C_b^{n+\theta}(\R^d)\big)$ if all
spatial partial derivatives $\nabla_{i_1}\ldots \nabla_{i_k}f\in L^\infty\big([0,T],C_b^\theta(\R^d)\big)$
for all $k=0,1,\ldots,n$, where $\nabla_j=\frac{\partial}{\partial x_j}$. The
corresponding norm is defined as
  $$\|f\|_{n+\theta,T}=\|f\|_{0,T}+\sum_{k=1}^n\|\nabla^k f\|_{0,T}+ [\nabla^n f]_{\theta,T},$$
in which we have extended the previous notations $\|\cdot\|_{0,T}$ and $[\,\cdot\,]_{\theta,T}$
to tensors. In the same way, we can define the spaces $L^\infty\big([0,T],C_b^{n+\theta}(\R^d,\R^d)\big)$ and $L^\infty\big([0,T],C_b^{n+\theta}(\R^d,\R^d\otimes\R^d)\big)$,
and the associated norms. We can also extend these function spaces to $T=\infty$
(we are considering functions defined on $[0,\infty)\times\R^d$), and the
corresponding norms are simply denoted by $\|\cdot\|_0$, $\|\cdot\|_{n+\theta}$ and so on.
These norms will also be used for functions in the spaces $C_b^{n+\theta}(\R^d)$,
$C_b^{n+\theta}(\R^d,\R^d)$ and $C_b^{n+\theta}(\R^d,\R^d\otimes\R^d)$ for all $n\ge0$, which are independent of time. There will be no confusion
according to the context.

We now recall the main result in \cite{Wang11} which will play an important role in the proof of
Theorem \ref{thm-Harnack}. To this end, let $\sigma:[0,\infty)\times\R^d\ra\R^d\otimes\R^d$ and
$b:[0,\infty)\times\R^d\to \R^d$ be two Borel measurable functions. We first list some assumptions
which are taken from \cite[Introduction]{Wang11}:
\begin{itemize}
\item[(A1)] for any $T>0$, there exists a constant $K_0>0$ such that
  $$\|\sigma(t,x)-\sigma(t,y)\|_{HS}^2+2\<b(t,x)-b(t,y),x-y\>\leq K_0|x-y|^2,\quad t\in[0,T],\, x,y\in\R^d;$$
\item[(A2)] for any $T>0$, there is a constant $\kappa_0>0$ such that
  $$a(t,x)=\sigma(t,x)\sigma(t,x)^\ast\geq \kappa_0^2\,\Id,\quad t\in[0,T],\, x\in\R^d;$$
\item[(A3)] for any $T>0$, there is a constant $\delta_0\ge0$ such that
  $$\big|(\sigma(t,x)-\sigma(t,y))(x-y)\big|\leq \delta_0|x-y|,\quad t\in[0,T],\, x,y\in\R^d.$$
\end{itemize}
It is well known that assumption (A1) ensures the pathwise uniqueness of solutions to \eqref{SDE}.
For the moment, we assume that SDE \eqref{SDE} has a unique strong solution $X_t$ and denote by $P_t$
the associated semigroup. F.-Y. Wang has shown in \cite[Theorem 1.1]{Wang11}
the following results on the Harnack inequalities for the semigroup $P_t$.

\begin{theorem}\label{sect-2-thm}
\begin{itemize}
\item[\rm(1)] If {\rm(A1)} and {\rm(A2)} hold, then
  $$P_t\log f(y)\leq \log P_t f(x)+\frac{K_0|x-y|^2}{2\kappa_0^2(1-e^{-K_0t})},\quad f\in\B_b(\R^d)
  \textrm{ with } f\ge 1,x,y\in\R^d, 0<t\le T.$$
\item[\rm(2)] If {\rm(A1)}, {\rm(A2)} and {\rm(A3)} hold, then for $p>(1+\delta_0/\kappa_0)^2$
and $\delta_p:=\max\{\delta_0,\kappa_0(\sqrt{p}-1)/2\}$,
  $$(P_t f(y))^p\leq (P_t f^p(x))\exp\bigg[\frac{K_0\sqrt{p}(\sqrt{p}-1)|x-y|^2}
  {4\delta_p[(\sqrt{p}-1)\kappa_0-\delta_p](1-e^{-K_0 t})}\bigg]$$
holds for all $f\in\mathcal B_b^+(\R^d)$, $x,y\in\R^d$ and $0<t\le T$.
\end{itemize}
\end{theorem}

\subsection{Proof of Theorem \ref{thm-Harnack}}

We need some more preparations. Fix any $T>0$.
We redefine $b$ and $\sigma$ on the space $[T,\infty)\times\R^d$ by setting
  $$b(t,x)=b(T,x)\mbox{ and } \sigma(t,x)=\sigma(T,x),\quad t\geq T,\,\ x\in \R^d.$$
Obviously, we have $b\in L^\infty\big([0,\infty),C_b^\theta(\R^d,\R^d)\big)$ and $\sigma\in L^\infty\big([0,\infty),C_b^{1+\theta}(\R^d,\R^d\otimes\R^d)\big)$. Given $\lambda>0$ and a function
$f\in L^\infty\big([0,\infty),C_b^\theta(\R^d,\R^d)\big)$, consider the equation
  \begin{equation}\label{resolvent-time}
  \partial_t u_\lambda + {L }_tu_\lambda - \lambda u_\lambda=f\quad\mbox{in } [0,\infty)\times\R^d,
  \end{equation}
where the operator ${L }_t$ associated to SDE \eqref{SDE} is defined by
  $${L }_tf(t,x)=\frac12\Tr[a(t,x)\nabla^2f(t,x)]+\<b(t,x),\nabla f(t,x)\>,$$
for some regular enough function $f:[0,\infty)\times\R^d\rightarrow \R^d$. Note that the
solution to \eqref{resolvent-time} is understood in the same sense as in \cite[p.10]{FGP}.

By the sketchy arguments of \cite[Remark 9]{Flandoli}, we have

\begin{lemma}\label{3-lem-1}
Assume that conditions {\rm(H3)} and {\rm(H4)} hold with some constant $\theta\in(0,1)$.
For any $f\in L^\infty\big([0,\infty),C_b^\theta(\R^d,\R^d)\big)$, there exists a
unique solution $u_\lambda$ to  equation \eqref{resolvent-time} in the space $L^\infty\big([0,\infty),C_b^{2+\theta}(\R^d,\R^d)\big)$ such that for any $\lambda\geq 1$,
  $$\|\nabla u_\lambda\|_0\leq
  \frac{C}{\sqrt\lambda}\|f\|_\theta,$$
where the constant $C>0$ is independent of $\lambda$.
\end{lemma}

Now for $\lambda>0$, consider the parabolic system
  \begin{equation}\label{auxiliary-eq}
  \partial_t \psi_\lambda + {L }_t\psi_\lambda - \lambda \psi_\lambda=b,
  \quad (t,x)\in [0,\infty)\times\R^d.
  \end{equation}
By Lemma \ref{3-lem-1}, there exists a unique solution $\psi_\lambda\in
L^\infty\big([0,\infty),C_b^{2+\theta}(\R^d,\R^d)\big)$ to the equation \eqref{auxiliary-eq}.
Define $\Psi_\lambda(t,x)=x+\psi_\lambda(t,x)$ for $t\geq0$ and $x\in \R^d$. Then we have

\begin{lemma}\label{3-lem-2}
For $\lambda$ large enough such that $\|\nabla\psi_\lambda\|_0<1$,
the following statements hold:
\begin{itemize}
\item[\rm(i)] $\Psi_\lambda$ has bounded first and second order spatial
derivatives uniformly in $t\in[0,\infty)$ and, moreover, the second order derivative $\nabla^2\Psi_\lambda$ is globally $\theta$-H\"older continuous uniformly in $t\in[0,\infty)$;
\item[\rm(ii)] for any $t\geq 0$, $\Psi_\lambda(t,\cdot)$ is a $C^2$-diffeomorphism of $\R^d$;
\item[\rm(iii)] $\Psi_\lambda^{-1}(t,\cdot)$ has bounded first and second order spatial derivatives
uniformly in $t\in [0,\infty)$ and, moreover,
  $$\nabla\Psi_\lambda^{-1}(t,\cdot)(x)=\sum_{k\geq0}\big[-\nabla\psi_\lambda(t,\Psi_\lambda^{-1}(t,\cdot)(x))\big]^k,
  \quad x\in\R^d, t\geq 0.$$
\end{itemize}
\end{lemma}

We choose $\lambda$ large enough such that
  \begin{equation*}\label{3-gradient}
  \|\nabla\psi_\lambda\|_0\leq \frac12.
  \end{equation*}
To simplify the notations, we shall omit the subscript $\lambda$ and write $\psi_t(x)$
(resp. $\Psi_t(x)$) instead of $\psi(t,x)$ (resp. $\Psi(t,x)$).
For $t\geq 0$ and $x\in\R^d$, set
  \begin{equation}\label{coefficients}
  \hat{\sigma}(t,x)=\nabla\Psi_t(\Psi_t^{-1}(x))\,\sigma(t,\Psi_t^{-1}(x)),\quad
  \hat{b}(t,x)=\lambda \psi_t(\Psi_t^{-1}(x)).
  \end{equation}
Consider the SDE on $\R^d$:
  \begin{equation}\label{conjugate-SDE}
  \d \hat X_t=\hat{\sigma}(t,\hat X_t)\,\d W_t+  \hat{b}(t,\hat X_t)\,\d t,\quad \hat X_0=y.
  \end{equation}
This equation is equivalent to \eqref{SDE} in the sense that if $X_t$ is a solution to
\eqref{SDE}, then $\hat X_t=\Psi_t(X_t)$ satisfies \eqref{conjugate-SDE} with $y=\Psi_0(x)$;
conversely, if $\hat X_t$ is a solution to \eqref{conjugate-SDE},
then $X_t=\Psi_t^{-1}(\hat X_t)$ solves \eqref{SDE} with $x=\Psi_0^{-1}(y)$.
From this we also deduce the relationship between their semigroups. Indeed, let $\hat P_t$
be the semigroup associated to \eqref{conjugate-SDE}, then for any $f\in\B_b(\R^d)$,
we have
  \begin{align}\label{relationship}
  P_tf(x)&=\E f(X_t(x))=\E\big[\big(f\circ\Psi_t^{-1}\big)\big(\Psi_t(X_t(x))\big)\big]\cr
  &=\E\big[\big(f\circ\Psi_t^{-1}\big)\big(\hat X_t(\Psi_0(x))\big)\big]
  =\hat P_t\big(f\circ\Psi_t^{-1}\big)(\Psi_0(x)).
  \end{align}

Now by Lemma \ref{3-lem-2}, we can verify that the assumptions (A1)--(A3) in Subsection 3.1
are satisfied by $\hat\sigma$ and $\hat b$. We collect the computations in the next lemma.

\begin{lemma}\label{sect-2-lem-3}
Under the hypotheses {\rm (H3)} and {\rm (H4)}, the coefficients
$\hat\sigma$ and $\hat b$ given by \eqref{coefficients} satisfy
the assumptions {\rm (A1)--(A3)} in Theorem $\ref{sect-2-thm}$. More
precisely, for any $T>0$, there exist positive constants $K_1,\kappa_1$ and
$\delta_1$ such that for all $0\le t\le T$ and $x$, $y\in\R^d$, it hold
\begin{itemize}
\item[\rm(1)] $\|\hat\sigma(t,x)-\hat\sigma(t,y)\|_{HS}^2+2\<\hat b(t,x)-\hat b(t,y),x-y\>\leq K_1|x-y|^2$;
\item[\rm(2)] $\hat a(t,x)=\hat\sigma(t,x)\hat\sigma(t,x)^\ast\geq \kappa_1^2\,\Id$;
\item[\rm(3)] $\big|(\hat\sigma(t,x)-\hat\sigma(t,y))(x-y)\big|\leq \delta_1|x-y|$.
\end{itemize}
\end{lemma}

\begin{proof}
(1) For $x,y\in\R^d$,
  \begin{align*}\label{sect-2-lem-3.1}
  \|\hat\sigma(t,x)-\hat\sigma(t,y)\|_{HS}
  &=\big\|\nabla\Psi_t(\Psi_t^{-1}(x))\,\sigma(t,\Psi_t^{-1}(x))
  -\nabla\Psi_t(\Psi_t^{-1}(y))\,\sigma(t,\Psi_t^{-1}(y))\big\|_{HS}\cr
  &\leq \big\|\nabla\Psi_t(\Psi_t^{-1}(x))-\nabla\Psi_t(\Psi_t^{-1}(y))\big\|_{HS}
  \|\sigma(t,\Psi_t^{-1}(x))\|_{HS}\cr
  &\hskip13pt +\|\nabla\Psi_t(\Psi_t^{-1}(y))\|_{HS}
  \big\|\sigma(t,\Psi_t^{-1}(x))-\sigma(t,\Psi_t^{-1}(y))\big\|_{HS}.
  \end{align*}
Since $\Psi_t(x)=x+\psi_t(x)$, we have $\nabla\Psi_t(x)=\Id+\nabla\psi_t(x)$, and
hence
  \begin{align*}
  \big\|\nabla\Psi_t(\Psi_t^{-1}(x))-\nabla\Psi_t(\Psi_t^{-1}(y))\big\|_{HS}
  &=\big\|\nabla\psi_t(\Psi_t^{-1}(x))-\nabla\psi_t(\Psi_t^{-1}(y))\big\|_{HS}\cr
  &\leq \|\nabla^2\psi_t\|_0|\Psi_t^{-1}(x)-\Psi_t^{-1}(y)|\\
  &\leq \|\nabla^2\psi_t\|_0 \|\nabla\Psi_t^{-1}\|_0|x-y|.
  \end{align*}
As $\sup_{t\geq0}\|\nabla\psi_t\|_0\leq 1/2$, the assertion (iii) in Lemma \ref{3-lem-2} implies that
  $$\|\nabla\Psi_t^{-1}\|_0\leq \sum_{k\geq0}\|\nabla\psi_t\|_0^k\leq 2.$$
Thus,
  \begin{equation*}\label{sect-2-lem-3.2}
  \big\|\nabla\Psi_t(\Psi_t^{-1}(x))-\nabla\Psi_t(\Psi_t^{-1}(y))\big\|_{HS}
  \leq 2\|\nabla^2\psi_t\|_0 |x-y|.
  \end{equation*}
Next, it holds that
  $$
  \aligned
  \big\|\sigma(t,\Psi_t^{-1}(x))-\sigma(t,\Psi_t^{-1}(y))\big\|_{HS}
  &\leq \|\nabla\sigma(t,\cdot)\|_0|\Psi_t^{-1}(x)-\Psi_t^{-1}(y)|\cr
  &\leq \|\nabla\sigma(t,\cdot)\|_0\|\nabla\Psi_t^{-1}\|_0|x-y|\\
  &\leq 2\|\nabla\sigma(t,\cdot)\|_0|x-y|.
  \endaligned
  $$
Therefore, we have
  $$
  \aligned
  \|\hat\sigma(t,x)-\hat\sigma(t,y)\|_{HS}
  &\leq 2\|\nabla^2\psi_t\|_0\|\sigma(t,\cdot)\|_0 |x-y|+ 2\|\nabla\Psi_t\|_0\|\nabla\sigma(t,\cdot)\|_0|x-y|\cr
  &=2\big(\|\nabla^2\psi_t\|_0\|\sigma(t,\cdot)\|_0+\|\nabla\Psi_t\|_0\|\nabla\sigma(t,\cdot)\|_0\big)|x-y|.
  \endaligned
  $$

On the other hand, by \eqref{coefficients} and the fact that $\|\nabla\psi_t\|_0\|\nabla\Psi_t^{-1}\|_0
\leq 1$ for all $t\geq0$,
  \begin{align*}
  |\<\hat b(t,x)-\hat b(t,y),x-y\>|&\leq |\hat b(t,x)-\hat b(t,y)|\, |x-y|\\
  &\leq \lambda\|\nabla\psi_t\|_0|\Psi_t^{-1}(x)-\Psi_t^{-1}(y)|\, |x-y|\\
  &\leq \lambda\|\nabla\psi_t\|_0\|\nabla\Psi_t^{-1}\|_0|x-y|^2\\
  &\leq \lambda|x-y|^2.
  \end{align*}
Combining all the estimates above, we get the desired estimate (A1) with
  $$K_1=4\sup_{t\geq0}\big(\|\nabla^2\psi_t\|_0\|\sigma(t,\cdot)\|_0
  +\|\nabla\Psi_t\|_0\|\nabla\sigma(t,\cdot)\|_0\big)^2+2\lambda<+\infty.$$

(2) For any $x,z\in\R^d$, it follows from the definition of $\hat\sigma$ that
  \begin{align*}
  \<\hat\sigma(t,x)\hat\sigma(t,x)^\ast z,z\>&=\big\<\nabla\Psi_t(\Psi_t^{-1}(x))\, \sigma(t,\Psi_t^{-1}(x))
  \,\sigma(t,\Psi_t^{-1}(x))^\ast\, [\nabla\Psi_t(\Psi_t^{-1}(x))]^\ast z,z\big\>\cr
  &=\big\<a(t,\Psi_t^{-1}(x))\, [\nabla\Psi_t(\Psi_t^{-1}(x))]^\ast z,[\nabla\Psi_t(\Psi_t^{-1}(x))]^\ast z\big\>.
  \end{align*}
Denote by $\lambda_i(t,x),1\leq i\leq d$, the eigenvalues of $a(t,x)$.
Under the hypothesis (H4), we have $\lambda_i(t,x)>0$ for all
$i\in\{1,\ldots,d\}$ and $(t,x)\in[0,\infty)\times \R^d$. By Cauchy's inequality,
  \begin{align*}
  \sum_{i=1}^d\frac 1{\lambda_i(t,x)}
  &=\Tr(a^{-1}(t,x))\leq \sqrt{d}\,\bigg(\sum_{i=1}^d(a_{ii}^{-1}(t,x))^2 \bigg)^{\frac12}\\
  &\leq \sqrt{d}\,\|a^{-1}(t,x)\|_{HS}\leq \sqrt{d}\,\|a^{-1}\|_0.
  \end{align*}
Hence
  $$\inf_{(t,x)\in[0,\infty)\times \R^d}\lambda_i(t,x)\geq \frac 1{\sqrt{d}\,\|a^{-1}\|_0}>0,\quad 1\leq i\leq d.$$
As a consequence,
  \begin{align*}
  \<\hat\sigma(t,x)\hat\sigma(t,x)^\ast z,z\>
  &\geq \frac 1{\sqrt{d}\,\|a^{-1}\|_0}\big|[\nabla\Psi_t(\Psi_t^{-1}(x))]^\ast z\big|^2.
  \end{align*}
Noting that $\nabla\Psi_t=\Id+\nabla\psi_t$, thus for any $y\in\R^d$,
  $$\big| [\nabla\Psi_t(y)]^\ast z\big|=\big|z+[\nabla\psi_t(y)]^\ast z\big|
  \geq |z|-\big| [\nabla\psi_t(y)]^\ast z\big|
  \geq |z|-\|\nabla\psi_t(y)\|_{HS}|z|\geq |z|/2,$$
thanks to $\|\nabla\psi_t\|_0=\sup_{y\in\R^d}\|\nabla\psi_t(y)\|_{HS}\leq 1/2$ for all $t\geq 0$.
Therefore
  $$\<\hat\sigma(t,x)\hat\sigma(t,x)^\ast z,z\>\geq \frac{|z|^2}{4\sqrt{d}\,\|a^{-1}\|_0},$$
which means that the condition (A2) holds with $\kappa_1=\big(4\sqrt{d}\,\|a^{-1}\|_0\big)^{-1/2}$.

(3) Since $\sigma$ is bounded, it is obvious that for any $(t,x)\in[0,\infty)\times \R^d$,
  $$\|\hat\sigma(t,x)\|_{HS}\leq \|\nabla\Psi_t(\Psi_t^{-1}(x))\|_{HS}\|\sigma(t,\Psi_t^{-1}(x))\|_{HS}
  \leq (d+\|\nabla\psi_t\|_0)\|\sigma\|_0\leq (d+1/2)\|\sigma\|_0.$$
Then the condition (A3) holds with $\delta_1=(2d+1)\|\sigma\|_0$.
\end{proof}

Now we are ready to prove Theorem \ref{thm-Harnack}.

\begin{proof}[Proof of Theorem \ref{thm-Harnack}]
By Lemma \ref{sect-2-lem-3} and Theorem \ref{sect-2-thm}, for any fixed $T>0$ we know that the semigroup $\hat P_t$
corresponding to the SDE \eqref{conjugate-SDE} satisfies the log-Harnack inequality:
  \begin{equation}\label{sect-1-thm-proof.1}
  \hat P_t\log g(y)\leq \log \hat P_t g(x)+\frac{K_1|x-y|^2}{2\kappa_1^2(1-e^{-K_1t})},
  \quad 0<t\le T,g\in \B_b(\R^d)\textrm{ with }g\geq 1, x,y\in\R^d,
  \end{equation}
and the Harnack inequality with power: for $p>(1+\delta_1/\kappa_1)^2$
and $\delta_p:=\max\{\delta_1,\kappa_1(\sqrt{p}-1)/2\}$, it holds
  \begin{equation*}\label{sect-1-thm-proof.2}
  (\hat P_t g(y))^p\leq (\hat P_t g^p(x))\exp\bigg[\frac{K_1\sqrt{p}(\sqrt{p}-1)|x-y|^2}
  {4\delta_p[(\sqrt{p}-1)\kappa_1-\delta_p](1-e^{-K_1 t})}\bigg]
  \end{equation*}
for all $0<t\le T$, $g\in\mathcal B_b^+(\R^d)$ and $x,y\in\R^d$. Now we shall apply the
relation \eqref{relationship} to show that the semigroup $P_t$ associated with the SDE
\eqref{SDE} also satisfies the same Harnack inequalities (possibly with different constants).

(1) Take $0<t\le T$ and $x,y\in\R^d$. For $f\in\B_b(\R^d)$ with $f\geq1$, set $g=f\circ\Psi_t^{-1}$.
We have by \eqref{relationship} and \eqref{sect-1-thm-proof.1},
  \begin{align*}
  P_t \log f(y)&=\hat P_t \log g(\Psi_0(y))
  \leq \log\hat P_t g(\Psi_0(x))+\frac{K_1|\Psi_0(x)-\Psi_0(y)|^2}{2\kappa_1^2(1-e^{-K_1t})}.
  \end{align*}
Since $|\Psi_0(x)-\Psi_0(y)|\leq |x-y|+|\psi_0(x)-\psi_0(y)|\leq 3|x-y|/2$,
again by \eqref{relationship}, we have
  $$P_t \log f(y)\leq \log P_t f(x)+\frac{2K_1|x-y|^2}{\kappa_1^2(1-e^{-K_1t})},
  \quad f\in \B_b(\R^d)\textrm{ with }f\geq 1.  $$

(2) In the same way, for $f\in\mathcal B_b^+(\R^d)$, let $g=f\circ\Psi_t^{-1}$. Then, for any $0<t\le T$,
  \begin{align*}
  (P_t f(y))^p=(\hat P_t g(\Psi_0(y)))&\leq [\hat P_t g^p(\Psi_0(x))]
  \exp\bigg[\frac{K_1\sqrt{p}(\sqrt{p}-1)|\Psi_0(x)-\Psi_0(y)|^2}
  {4\delta_p[(\sqrt{p}-1)\kappa_1-\delta_p](1-e^{-K_1 t})}\bigg]\cr
  &\leq [P_t f^p(x)]\exp\bigg[\frac{K_1\sqrt{p}(\sqrt{p}-1)|x-y|^2}
  {\delta_p[(\sqrt{p}-1)\kappa_1-\delta_p](1-e^{-K_1 t})}\bigg].
  \end{align*}
The proof is complete.
\end{proof}

\section{Harnack Inequalities for SDEs Driven by Symmetric Stable Processes}
In this section, we consider the following SDE
  \begin{equation}\label{stable1}
  \d X_t=b(X_t)\,\d t+\d Z_t,
  \end{equation}
where $Z_t$ is a symmetric $\alpha$-stable process with $\alpha\ge
1$ and $b\in C_b^\beta(\R^d;\R^d)$ with $\beta>1-\frac{\alpha}{2}$.
In \cite[Theorem 1.1]{Priola}, Priola proved that the solution $X_t$ to \eqref{stable1}
is still a flow of homeomorphisms which are $C^1$-functions on $\R^d$.
His proof, similar to that in \cite{Flandoli}, is
based on the so-called It\^o--Tanaka trick which transforms the original SDE into a new one with
regular drift coefficient; see \cite[(4.10) or (4.14)]{Priola}. We tried to establish
Harnack inequalities for the semigroup $(P_t)_{t\ge0}$ associated to SDE
\eqref{stable1} in this framework, but in vain. The difficulty comes from the term
on the right hand side of \cite[(4.14)]{Priola}, which involves an integral with respect to the
Poisson random measure.

A possible approach to bypass the difficulty is to use the regularization
approximations of the underlying subordinator; see \cite{WW14} for the recent study of dimension-free Harnack inequalities for a class of stochastic equations driven by a L\'{e}vy noise containing a subordiante Brownian motion.
In this case, since the vector field $b$ is time-independent, the drift coefficient
in the new SDE is variable-separated,
and the diffusion coefficient is the identity matrix (cf. \cite[(2.3)]{Zhang13b} or \cite[(3.2)]{WW14}).
At first glance, it seems that one can establish the Harnack inequalities in the same way as
above by using the transform in \cite[page 14]{FGP}. However, this requires an explicit
expression for the constant $C$ in the Schauder estimate (see \cite[Theorem 2]{FGP}), especially
its dependence on the time-change factor. After checking the proofs of
\cite[Theorem 2 and Lemma 4]{FGP}, unfortunately, we did not obtain useful estimates on $C$.

On the other hand, by using the explicit heat kernel estimates, we can prove

\begin{proposition}There
exists a constant $C>1$ such that for any $f\in \B_b^+(\R^d)$, $T>0$
and $x,y\in\R^d$,
\begin{equation}\label{stable} P_Tf(x)\le
 C\left(1+\frac{|x-y|}{(T\wedge1)^{1/\alpha}}\right)^{d+\alpha}P_Tf(y).\end{equation}\end{proposition}

\begin{proof}
Note that the drift term $b$ is bounded, and so it belongs to the
Kato class (see e.g. \cite[Definition 1.1]{CW}). Then, for
$\alpha>1$, according to \cite[Corollary 1.3]{CW}, the process $X_t$
has a jointly continuous transition density function $p(t,x,y)$ with
respect to the Lebesgue measure on $\R^d$. Moreover, there exists a
constant $c\ge 1$ such that  for all $(t,x,y)\in (0,1]\times
\R^d\times \R^d$,
\begin{equation}\label{den1}c^{-1}\left(t^{-d/\alpha}\wedge\frac{t}{|x-y|^{d+\alpha}}\right)\le
p(t,x,y)\le
c\left(t^{-d/\alpha}\wedge\frac{t}{|x-y|^{d+\alpha}}\right).\end{equation}
On the other hand, noticing that $b\in C_b^\beta(\R^d;\R^d)$ with
$\beta>1-\frac{\alpha}{2}$, we know from \cite[Theorem 1.1]{XZ} (see
also \cite[Theorem 3.6]{XZ}) that \eqref{den1} also holds for
$\alpha=1$.

Furthermore, having \eqref{den1} at hand and following the proof of
\cite[Lemma 2.1]{W}, we can get that for any $t\in (0,1]$ and
$x,y,z\in \R^d$,
\begin{equation}\label{den2}\frac{p(t,x,z)}
{p(t,y,z)}\le
2^{\alpha+d}c^2\left(1+\frac{|x-y|}{t^{1/\alpha}}\right)^{d+\alpha}.\end{equation}
Therefore, for any $f\in \B_b^+(\R^d)$, $t\in (0,1]$ and
$x,y\in\R^d$,
\begin{align*}
  P_tf(x)&=\int f(z)p(t,x,z)\,dz=\int f(z)\frac{p(t,x,z)}{p(t,y,z)}p(t,y,z)\,dz\\
  &\le \left(\max_{z\in\R^d}\frac{p(t,x,z)}{p(t,y,z)}\right) \int f(z)p(t,y,z)\,dz\\
  &=C\left(1+\frac{|x-y|}{t^{1/\alpha}}\right)^{d+\alpha}P_tf(y),
  \end{align*}
where $C$ is a positive constant independent of $x,y,z$ and $t$. For
any $t>0$, we write
$$P_tf=P_{t\wedge1}P_{(t-1)^+}f.$$ This, together with the inequality
above, yields the required assertion.
\end{proof}

Surely, the Harnack inequality \eqref{stable} is not satisfactory in the sense that
$C>1$, which means that such inequality is not sharp for the case
$x=y$. Nonetheless, since the process has the transition density
function, it has the strong Feller property, and so even for
$C>1$, we still have some applications, e.g., long time
behaviors and properties of invariant measure.


\begin{thebibliography}{99}
\bibitem{ArnaudonThalmaierWang2006} M. Arnaudon, A. Thalmaier, F.-Y. Wang, Harnack inequality and
heat kernel estimates on manifolds with curvature unbounded below. {\it Bull. Sci. Math.} 130 (2006), no. 3, 223--233.

\bibitem{BGL} D. Bakry, I. Gentil, M. Ledoux,  \emph{Analysis and geometry of Markov diffusion
operators}, Grundlehren der Mathematischen Wissenschaften 348, Springer, 2014.

\bibitem{BFGM} L. Beck, F. Flandoli, M. Gubinelli, M. Maurelli, Stochastic ODEs and stochastic
linear PDEs with critical drift: regularity, duality and uniqueness. arXiv:1401.1530v1

\bibitem{CW} Z.-Q. Chen, L. Wang, Uniqueness of stable processes with drift.
arXiv:1309.6414v1

\bibitem{FF13} E. Fedrizzi, F. Flandoli, H\"older flow and differentiability for SDEs
with nonregular drift. {\it Stoch. Anal. Appl.} 31 (2013), no. 4, 708--736.

\bibitem{FGP} F. Flandoli, M. Gubinelli, E. Priola, Well-posedness of the transport equation by
stochastic perturbation. {\it Invent. Math.} 180 (2010), no. 1, 1--53.

\bibitem{Flandoli} F. Flandoli, M. Gubinelli, E. Priola, Flow of diffeomorphisms for SDEs
with unbounded H\"older continuous drift. {\it Bull. Sci. Math.} 134 (2010), no. 4, 405--422.

\bibitem{KR05} N.V. Krylov, M. R\"ockner, Strong solutions of stochastic equations with singular
time dependent drift. {\it Probab. Theory Related Fields} 131 (2005), no. 2, 154--196.

\bibitem{Priola} E. Priola, Pathwise uniqueness for singular SDEs driven by stable processes.
{\it Osaka J. Math.} 49 (2012), no. 2, 421--447.

\bibitem{Rockner} M. R\"ockner, F.-Y. Wang, Log-Harnack inequality for stochastic differential
equations in Hilbert spaces and its consequences. {\it Infin. Dimens. Anal. Quantum Probab.
Relat. Top.} 13 (2010), no. 1, 27--37.

\bibitem{Shao} J. Shao, Harnack inequalities and heat kernel estimates for SDEs with singular
drifts. {\it Bull. Sci. Math.} 137 (2013), no. 5, 589--601.

\bibitem{SWY} J. Shao, F.-Y. Wang, C. Yuan, Harnack ineuqlaities for stochastic (functional)
differential equations with non-Lipschitzian coefficients. {\it Electron. J. Probab.} 17 (2011), Paper no. 100, 1--18.

\bibitem{Wang97} F.-Y. Wang, Logarithmic Sobolev inequalities on noncompact Riemannian manifolds.
{\it Probab. Theory Related Fields} 109 (1997), no. 3, 417--424.

\bibitem{Wang11} F.-Y. Wang, Harnack inequality for SDE with multiplicative noise
and extension to Neumann semigroup on nonconvex manifolds. {\it Ann. Probab.} 39 (2011),
no. 4, 1449--1467.

\bibitem{Wang-book} F.-Y. Wang, \emph{Harnack Inequalities for Stochastic Partial Differential Equations},
Springer, New York, 2013.

\bibitem{WW14} F.-Y. Wang, J. Wang, Harnack inequalities for stochastic equations driven by L\'{e}vy niose. {\it
J. Math. Anal. Appl.} 410 (2014), no. 1, 513--523.

\bibitem{W} J. Wang, Harnack inequalities for Ornstein--Uhlenbeck processes driven by L\'{e}vy processes. {\it
Stat. Probab. Letters} 81 (2011), no. 9, 1436--1444.


\bibitem{XZ} L. Xie, X. Zhang, Heat kernel estimates for critical fractional
operators. arXiv: 1210.7063v2

\bibitem{Zhang05b} X. Zhang, Strong solutions of SDEs with singular drift and Sobolev
diffusion coefficients. {\it Stoch. Proc. Appl.} 115 (2005), no. 11, 1805--1818.

\bibitem{Zhang11} X. Zhang, Stochastic homeomorphism flows of SDEs with singular drifts
and Sobolev diffusion coefficients. {\it Electron. J. Probab.} 16 (2011), no. 38, 1096--1116.

\bibitem{Zhang13b} X. Zhang, Derivative formulas and gradient estimates for SDEs driven
by $\alpha$-stable processes. {\it Stoch. Proc. Appl.} 123 (2013), no. 4, 1213--1228.

\end{thebibliography}
\end{document}